\documentclass[11pt,a4paper,leqno]{amsart}
\usepackage[utf8]{inputenc}
 \nonstopmode \textwidth=16.00cm
\usepackage{amssymb,color,tikz, cancel}
\usepackage{amsmath,amscd,amsfonts,amsthm,verbatim}
\usepackage{tikz-cd}
\usepackage{enumitem}
\usepackage[all]{xy}
\usetikzlibrary{matrix}

\usepackage{xcolor}
\usepackage{blindtext}
\usepackage{multicol}

\numberwithin{equation}{section} 

\definecolor{ffzzqq}{rgb}{1,0.6,0}
\definecolor{qqqqff}{rgb}{0,0,1}
\definecolor{ffqqqq}{rgb}{1,0,0}
\definecolor{wwzzqq}{rgb}{0.4,0.6,0}
\definecolor{zzwwff}{rgb}{0.6,0.4,1}
\definecolor{ttqqqq}{rgb}{0.2,0,0}

\usepackage{pgfplots}
\pgfplotsset{compat=1.15}
\usepackage{mathrsfs}
\usetikzlibrary{arrows}

\newcommand{\CC}{\mathbb{C}}
\newcommand{\mA}{\mathbb{A}}
\newcommand{\NN}{\mathbb{N}}

\newcommand{\ZZ}{\mathbb{Z}} 
\newcommand {\PP}{\mathbb{P}}
\newcommand {\mP}{\mathbb{P}}
\newcommand {\sI}{\mathcal{I}}

\newcommand {\sO}{\mathcal{O}}

\newcommand {\sL}{\mathcal{L}}
\newcommand {\sC}{\mathcal{C}}
\newcommand {\sE}{\mathcal{E}}
\newcommand {\sF}{\mathcal{F}}
\newcommand {\sG}{\mathcal{G}}

 \DeclareMathOperator{\Proj}{Proj}
 \def\cocoa{{\hbox{\rm C\kern-.13em
   o\kern-.07em C\kern-.13em o\kern-.15em A}}}
\newtheorem{theorem}{Theorem}[section]
\newtheorem{lemma}[theorem]{Lemma}
\newtheorem{proposition}[theorem]{Proposition}
\newtheorem{corollary}[theorem]{Corollary}
\newtheorem{question}[theorem]{Question}
\newtheorem{conjecture}[theorem]{Conjecture} \theoremstyle{definition}
\newtheorem{definition}[theorem]{Definition} \theoremstyle{remark}
\newtheorem{remark}[theorem]{Remark}
\newtheorem{example}[theorem]{Example}

\definecolor{MyDarkGreen}{cmyk}{0.7,0,1,0}

\begin{document}

\title[Quasi-homogeneous singularities of free and nearly free plane curves]{A characterization of quasi-homogeneous singularities of free and nearly free plane curves}
 \author[A. V. Andrade]{Aline V. Andrade}
 \address{ICEx - UFMG, Department of Mathematics, Av. Ant\^onio Carlos, 6627, 30123-970 Belo Horizonte, MG, Brazil}
 \email{andradealine@mat.ufmg.br, ORCID 0000-0001-7129-3953}
 \author[V. Beorchia]{Valentina Beorchia} 
 \address{Dipartimento di Matematica, Informatica e Geoscienze, Universit\`a di
Trieste, Via Valerio 12/1, 34127 Trieste, Italy}
 \email{beorchia@units.it, 
 ORCID 0000-0003-3681-9045}
 \author[R.\ M.\ Mir\'o-Roig]{Rosa M.\ Mir\'o-Roig} 
 \address{Facultat de
 Matem\`atiques i Inform\`atica, Universitat de Barcelona, Gran Via des les
 Corts Catalanes 585, 08007 Barcelona, Spain} \email{miro@ub.edu, ORCID 0000-0003-1375-6547}

\thanks{The first author is supported by CAPES/Print grant number 88887.913035/2023-00 and partially supported by CNPq universal grant 408974/2023-0.}
\thanks{The second author is a member of GNSAGA of INdAM, is supported by MUR funds: PRIN project GEOMETRY OF ALGEBRAIC STRUCTURES: MODULI, INVARIANTS, DEFORMATIONS, PI Ugo Bruzzo, Project code 2022BTA242, and by the University of Trieste project FRA 2024.} 
\thanks{The third author has been partially supported by the grant PID2020-113674GB-I00}


\medskip \noindent

\begin{abstract} The goal of this paper is to establish a new characterization of quasi-homogeneous isolated singularities of free curves and nearly free curves $C$ in $\PP_\CC^2$. The criterion will be in terms of a first syzygy matrix associated with the Jacobian ideal $J_f$ of $f$, where $f=0$ is the equation of the plane curve $C$.
\end{abstract}

\maketitle

\section{Introduction}

The present paper concerns the question whether a singular point of a reduced planar curve $V(f) \subset \PP^2: =\PP_\CC^2$ is weighted homogeneous or not, that is when an isolated singular point can be locally represented by a weighted homogeneous polynomial. We recall that
a polynomial $f\in \CC[x_1,x_2]$ is called
{\it weighted homogeneous} if there exist positive rational numbers 
$w_1,w_2$ and $\delta$ such that, for each monomial $x_1^{a_1}x_2^{a_2}$ appearing in $f$ with nonzero
coefficient, one has $a_1 w_1 +a_2w_2=\delta$. The number $\delta$ is called the weighted homogeneous degree of $f$ with respect to the weights $w_1,w_2$.
A celebrated result by Saito (see \cite{S})
gives the following characterization of weighted homogeneous singularities.

Let $g = g(x_1,x_2)$ be a holomorphic function in a neighborhood of the origin in 
$\CC^2$ defining an isolated singularity at $(0,0)$. The following conditions are equivalent:
\begin{enumerate}
\item there is a holomorphic coordinate transformation mapping $g$ to a weighted homogeneous polynomial;
\item $g(x_1,x_2)$ is quasi-homogeneous, i.e., there exist holomorphic local germs \[a_1(x_1,x_2),a_2(x_1,x_2)\] 
at $(0,0)$ such that
$g(x_1,x_2)=a_1(x_1,x_2)\frac{\partial g}{\partial x_1}+ a_2(x_1,x_2)\frac{\partial g}{\partial x_2}$;
\item it holds $\mu_{(0,0)}(g)=\tau_{(0,0)}(g)$, where $\mu$ and $\tau$ are the Milnor and the Tjurina numbers of the isolated singularity
defined by $g$ (see Definition \ref{def_tau_mu}).
\end{enumerate}

The definition and the characterization above turn out to be difficult to test in concrete cases, especially for singular points of high multiplicity.
In \cite{YZ} the following problem is addressed: How many different ways can the weighted homogeneous isolated hypersurface singularities of a polynomial be recognized? 

The paper \cite{YZ} illustrates nine different ways for such a characterization, reflecting the fact that this topic plays an important role in singularity theory.

Here we introduce an additional characterization, which concerns singularities of free and nearly free plane curves. We recall that a projective plane curve $V(f)$, where $f\in R:=\CC[x_0,x_1,x_2]$, is said to be free if, by setting $J_f=\langle \partial_{x_0} f, \partial_{x_1} f, \partial_{x_2} f \rangle$ to be
the Jacobian ideal of $f$,
the graded $R$-module ${\rm Syz}(J_f)$
of all Jacobian relations for $f$ is a free $R$-module. This is equivalent to saying that the
projective dimension ${\rm pd} (R/J_f) =2$; by the 
Hilbert Syzygy Theorem, see \cite[Corollary 19.7]{e}, and the Auslander-Buchsbaum formula, see \cite[Theorem 19.9]{e}, this is equivalent to 
the Jacobian ideal being saturated. A plane curve is called nearly free (see \cite[Definition 2.1]{DS5}) if the module $N(f)=\frac{I_f}{J_f}$, where $I_f$ is the saturation of $J_f$, satisfies $N(f) \neq 0$ and $\dim N(f)_s \le 1$ for any $s \in \ZZ$. The study of free curves is a very active area of research, also in connection with Terao's conjecture for line arrangements, which says that the freeness of a line arrangement only depends on the combinatorics of the intersection lattice (see \cite[Conjecture 4.138]{OT}). Moreover, by \cite[Corollary 1.2 and Theorem 1.3]{D}, free curves admitting a Jacobian syzygy in degree $r < \frac{d}{2}$ are characterized by attaining the maximal bound $\tau(d,r)_{max}$
of Du Plessis and Wall for the global Tjurina number $\tau(f)$, and nearly free curves
with $r \le \frac{d}{2}$ satisfy $\tau(f) = \tau(d,r)_{max} -1$, where $r$ denotes the minimal degree of a non-trivial Jacobian relation, see Proposition \ref{bounds_tau}.

The criterion we prove involves the global Jacobian singular scheme, that is the scheme defined by the Jacobian ideal $J_f$, and relies on a first syzygy matrix appearing in a minimal free resolution of $J_f$. Our main result is the following:

\begin{theorem}
 Let $C=V(f) \subset \PP^2$ be a free, respectively nearly free, curve of degree $d > 2r$, respectively $d \ge 2r$, where $r=mdr(f)$ is the minimal degree of a Jacobian syzygy (see Definition \ref{def_mdr}). Moreover, let $M_f$ be a first syzygy matrix in a minimal free resolution of the Jacobian ideal $J_f$.

 Then a singular point $p \in {\rm Sing}\ C$ is quasi-homogeneous if and only if ${\rm rk}\ M_f(p) \ge 1$.
 \end{theorem}

Our techniques are geometric; indeed, we consider the polar map
$\nabla f: \PP^2 \dasharrow \PP^2$, defined by the three partials of $f$, and the surface $S_f$ given by closure of its graph in $\PP^2 \times \PP^2$, or in other words the blow-up of the Jacobian scheme. It turns out that such a surface is contained in a locus
$Z_f$ with equations determined by the columns of $M_f$, and the quasi-homogeneity of the singular points is equivalent to the irreducibility of $Z_f$. 

We point out that our criterion is very simple to test; indeed, a first syzygy matrix is easily found by many symbolic computational programs, like Macaulay 2 \cite{M2} or Singular \cite{Sing}. Moreover, we can check if all the singular points of a given curve are quasi-homogeneous just by checking if the zero set of the ideal generated by the entries of $M_f$ is empty, that is its radical is irrelevant. Such a test can be done again with the aid of Macaulay 2, without explicitly computing the singular points. 

Furthermore, our result allows us to interpret the locus of first syzygy matrices of free curves admitting some non quasi-homogeneous singularity as a 
closed subscheme of the space of matrices.
Indeed, 
 the condition ${\rm rk}\ M_f=0$ is a closed condition on the coefficients of the entries of 
 $M_f$. Moreover, such a closed locus is proper, since there exist free line arrangements in any degree, and it is known that line arrangements have only quasi-homogeneous singularities (see \ref{eq: ex of free with only qh}), thus with
 nowhere vanishing matrix $M_f$.

We point out that all known characterizations are of {\it local} type, meaning that to apply them we need to know exactly the coordinates of a singular point and a local equation of the curve in such a point, while our approach allows to obtain information on the nature of the singularities by looking globally to the Jacobian singular scheme, so we think that this gives a new perspective to such a study. The investigation of a similar criterion for other classes of singular curves is a challenging question and deserves further study. 
In the case of $3$-syzygy curves, that is curves admitting exactly $3$ independent Jacobian syzygies, we conjecture that the same result as in the nearly free case holds, see Conjecture \ref{conj}. This expectation is motivated by a large number of numerical experiments.

Finally, we observe that our approach is related to the notions of Rees algebra $R(J_f)=\bigoplus _{n=0}^\infty J_f ^n$ and of symmetric algebra $S(J_f)=\bigoplus _{n=0}^\infty {\rm Sym}^n J_f$ of the Jacobian ideal
$J_f$. Indeed, we can interpret $R(J_f)$ as the coordinate ring of the blow-up of the Jacobian scheme, thus of the closure $S_f \subset \mP^2 \times \mP^2$ of the graph of the polar map, and $S(J_f)$ as the coordinate ring of
the surface $Z_f\subset \mP^2 \times \mP^2$ determined by a first syzygy matrix, see for instance \cite[Section 1.2]{ST2}. It has been proved in \cite[Lemma 3.1 and Corollary 3.2]{NS}, that $R(J_f) \cong S(J_f)$, and hence $S_f=Z_f$, if and only if the matrix
$M_f$ never vanishes; the property $R(J_f) \cong S(J_f)$ is referred as $J_f$ being of {\it linear type}. A rich research area concerns the question whether $R(J_f)$ and $S(J_f)$ are Cohen-Macaulay, see for instance \cite{BMNR}, \cite{BS}, \cite{Po}, and \cite{R}, just to quote a few of them. Our results can be translated as follows: we prove that the symmetric algebra of the Jacobian ideal of a free or nearly free curve is always Cohen-Macaulay, and that the Jacobian ideal is of linear type if and only if all the singularities are quasi-homogeneous.

We conclude this circle of comments by observing that, as far as we know, Rees and symmetric algebras have been involved in the study of Jacobian ideals only for families of curves with all quasi-homogeneous singularities, for instance in \cite {ST2} and \cite{T}. Thus our results give a new insight also in such a framework.

Now we illustrate
the organization of the paper. In Section 2 we fix the notation and the preliminary results. In Section 3 we prove our result in the case of free curves (see Theorem \ref{main1}), and in the next one, by an appropriate use of the de Rham sequence, we give a description of first syzygy matrices of free curves. In Section 5 we prove our result for nearly free curves (see Theorem \ref{main2}) and we describe a first syzygy matrix of nearly free curves with a linear Jacobian syzygy of a particular type (see Proposition \ref{1st_syz_NF}). In Section 6 we apply our criterion to several examples, detecting all non quasi-homogeneous singularities of conic-line arrangements with a linear Jacobian syzygy, which have been classified in \cite{BMR}. Finally, in Section 7 we conjecture that the same result holds for $3$-syzygy curves.
\vskip 4mm
\noindent
 {\bf Acknowledgement.} 
Most of this work was done while the first and third authors were visiting the Universit\`a di Trieste, and they would like to thank the Department of Mathematics, Computer Science and Geosciences for their warm hospitality. The authors thank the anonymous referee for useful comments. 
 
\section{Preliminaries } 
This section contains the basic definitions and results on Jacobian ideals associated with reduced singular plane curves and it lays the groundwork for the results in the later sections.

\vskip 2mm

\subsection{Jacobian ideal of a reduced curve}
From now on, we fix the polynomial ring $R=\CC[x_0,x_1,x_2]$ and we denote by $C =V(f)$ a reduced plane curve of degree $d$ defined by a homogeneous polynomial $f\in R_d$ in the
complex projective plane $\PP^2=\Proj(R)$.

\begin{definition} \label{jabscheme}
 The {\em Jacobian ideal} $J_f$ of 
a reduced singular plane curve $C =V(f)$ of degree $d$ is defined as the homogeneous ideal in $R$ generated by the 3 partial derivatives
$\partial_0 f:=\partial_{x_0} f$, $\partial_1 f:=\partial_{x_1} f$ and $\partial_2 f:=\partial_{x_2} f$. The {\em Jacobian scheme} $\Sigma_f$ of $C=V(f) \subseteq \PP^2$ is the zero-dimensional scheme defined by the homogeneous ideal $J_f$.
\end{definition}

We denote by ${\rm Syz}(J_f)$ the graded $R$-module of all Jacobian relations for $f$, i.e.,
$$
{\rm Syz}(J_f):=\{(a,b,c)\in R ^3\mid a\partial_0 f + b\partial_1 f+ c\partial_2 f = 0 \};
$$
and by ${\rm Syz}(J_f)_t$ the homogeneous part of degree $t$ of the graded $R$-module ${\rm Syz}(J_f)$. For any $t \ge 0$, we have that ${\rm Syz}(J_f)_t$ is a $\CC$-vector space of finite dimension.

\begin{definition} \label{def_mdr}
The {\em minimal degree of a Jacobian syzygy} for $f$ is the integer ${\rm mdr}(f)$ defined to be the smallest integer such that there
is a nontrivial relation
$ a\partial_0 f + b\partial_1 f + c\partial_2 f= 0$ 
with coefficients $a, b, c \in R_r$. More precisely, we have:
\begin{equation}
 \label{mdr}
{\rm mdr}(f):=min\{ n\in \NN \mid {\rm Syz}(J_f)_n\ne 0\}.
\end{equation}
\end{definition}

It is well known that ${\rm mdr}(f) = 0$, i.e. the three partials
$\partial_0 f$, $\partial_1 f$ and $\partial_2 f$ are linearly dependent, if and only if $C$ is
the union of lines passing through one point $p\in \PP^2$. So we will always assume that ${\rm mdr}(f)>0$.

\begin{definition}\label{def: free curves}
Let $C =V(f)$ be a reduced singular plane curve of degree $d$. We say that $C$ is {\em free} if the graded $R$-module ${\rm Syz}(J_f)$ of all Jacobian relations for $f$ is a free $R$-module, i.e. 
\begin{equation}\label{free}
{\rm Syz}(J_f)=R(-d_1)\oplus R(-d_2)
\end{equation}
with $d_1+d_2=d-1$. In this case $(d_1,d_2)$ are called the {\it exponents} of $C$.

We say that $C=V(f)$ is {\em nearly free} if 
if the module $N(f)=\frac{I_f}{J_f}$, where $I_f$ is the saturation of $J_f$, satisfies $N(f) \neq 0$ and $\dim N(f)_s \le 1$ for any $s \in \ZZ$. 
By \cite[Theorem 2.2]{DS5} and 
\cite[Definition 2.2]{D}, a curve $V(f)$ is nearly free if and only if
the minimal free resolution of ${\rm Syz}(J_f)$ looks like:
\begin{equation}\label{nearly free}
0 \longrightarrow R(-d-d_2) \longrightarrow R(1-d-d_1)\oplus R(1-d-d_2)^2 \longrightarrow {\rm Syz}(J_f) \longrightarrow 0
\end{equation}
with $d_1\le d_2$ and $d_1+d_2=d$.
\end{definition}

\begin{example}(1) The rational cuspidal septic $C\subset \PP^2$ of equation $C=V(f)=V(x_0^7+x_0^3x_1^4+x_1^6x_2)$ is free. Indeed, $J_f=(7x_0^6+3x_0^2x_1^4,4x_0^3x_1^3+6x_1^5x_2,x_1^6)\subset R$ and it has a minimal free $R$-resolution of the following type:
$$
0 \longrightarrow R(-9)^2 \longrightarrow R(-6)^3 \longrightarrow J_f \longrightarrow 0. $$
We have ${\rm mdr}(f)=3$, $\deg(J_f)=27$ and $C$ is free.

\noindent (2) The rational cuspidal septic $C\subset \PP^2$ of equation $C=V(f)=V(x_0^7+x_0^4x_1^3+x_1^6x_2)$ is nearly free. Indeed, $J_f=(7x_0^6+4x_0^3x_1^3,3x_0^4x_1^2+6x_1^5x_2,x_1^6)
\subset R$ and it has a minimal free $R$-resolution of the following type:
$$
0 \longrightarrow R(-11) \longrightarrow R(-9)\oplus R(-10)^2 \longrightarrow R(-6)^3 \longrightarrow J_f \longrightarrow 0. 
$$
We have ${\rm mdr}(f)=3$, $\deg (J_f)=26$ and $C$ is nearly free.

\vskip 2mm

\noindent (3) The rational septic $C\subset \PP^2$ of equation $C=V(f)=V(x_0^7+x_0^6x_2+x_1^6x_2)$ is neither free nor nearly free. Indeed, $J_f=(7x_0^6+6x_0^5x_2,6x_1^5x_2,x_0^6+x_1^6)\subset R$ and it has a minimal free $R$-resolution of the following type:
$$
0 \longrightarrow R(-13)^2\longrightarrow R(-8)\oplus R(-12)^3 \longrightarrow R(-6)^3 \longrightarrow J_f \longrightarrow 0. 
$$
We have ${\rm mdr}(f)=2$ and $\deg (J_f)=25$. 
\end{example}

Observe that the condition that a reduced singular curve $C=V(f)$ in $\PP^2$ is free is
equivalent to the Jacobian ideal $J_f$ being arithmetically Cohen-Macaulay
of codimension two; such ideals are completely described by the
Hilbert-Burch Theorem (see, for instance, \cite{e}): if $I = \langle g_1,\cdots,
g_m\rangle \subset R$ is a Cohen-Macaulay ideal of codimension two, then $I$ is defined
by the maximal minors of the $(m+1)\times m$ matrix of the first
syzygies of the ideal $I$. Combining this with Euler's formula for a homogeneous
polynomial, we get that a free curve $C=V(f)$ in $\PP^2$ has a very
constrained structure: there exists a $3 \times 3$ matrix $M$, with one row consisting of the 3 variables, and the remaining $2$ rows a minimal set of generators of the graded $R$-module ${\rm Syz} (J_f)$, such that $\det (M)\equiv 0 \ \textrm{mod  }(f)$.

\vskip 2mm
Let us recall some notions from singularity theory as commonly used.
We fix $C=V(f)\subset \PP^2$ a reduced, not necessarily irreducible, plane curve and fix an
isolated singular point $p\in C$. Up to a projectivity, we may assume $p=(1:0:0)$.

Let $g = g(x_1,x_2)$ be a holomorphic function in a neighborhood of the origin in 
$\mA^2_\CC =\CC^2$ defining the singularity in the affine open $x_0\neq 0$.

\begin{definition} \label{def_tau_mu}
The {\em Milnor number} of a reduced plane curve $C=V(f)\subset \PP^2$ at $p=(1:0:0)\in C$ is 
\begin{equation}\label{eq: local Milnor}
\mu _{p}(C) = \dim \CC \{x_1, x_2 \}/ \langle \partial_1 g,\partial_2 g \rangle.
\end{equation}

The {\em Tjurina number} of a reduced plane curve $C=V(f)\subset \PP^2$ at $p=(1:0:0)\in C$ is 
\begin{equation}\label{eq: local Tjurina}
\tau _{p}(C) = \dim \CC \{x_1, x_2 \}/ \langle \partial_1 g,\partial_2 g ,g\rangle.
\end{equation}

We clearly have $\tau _{p}(C)\le \mu _{p}(C)$ for any singular point $p$.
\end{definition}

\begin{definition}
A singularity is {\it quasi-homogeneous} if and only if there exists a holomorphic change of variables so that the local defining equation becomes weighted homogeneous. Recall that $g(x_1,x_2) =\sum c_{ij}x_1^ix_2^j$ is said to be {\em weighted homogeneous} if there exist rational numbers $\alpha,\beta$ such that $c_{ij}x_1^{i\alpha}x_2^{j\beta}$ is homogeneous.
\end{definition}

We define
$$
\mu (C):= \sum _{p\in Sing(C)} \mu _p(C)
$$
to be the {\em total Milnor number}.

For a projective plane curve $C=V(f)\subset \PP^2$, it holds:
$$
\tau (C):=\deg J_f = \sum _{p\in Sing(C)} \tau _p(C),
$$
where $J_f$ is the Jacobian ideal. We call $\tau (C)$ the {\em total Tjurina number} of $C$.

\vskip 2mm
Following Saito's classification of weighted homogeneous polynomials \cite{S} it follows that a reduced singular curve $C$ has a quasi-homogeneous singularity at $p$ if, and only if, $\mu_p(C)=\tau_p(C)$.
 
\begin{example}\label{singW13} Consider the following curves presented in \cite{Wall}. The notation used for the singularities is taken from Arnol'd's classification \cite{Arnold}.
\begin{enumerate} 
  
  \item A unicuspical quartic with a tangent in the cuspidal point (see \cite[Case Q10 with $t=0$, page 271]{Wall}, given by the equation $$C=V(f)=V(x_0(x_0^3x_2+x_1^4))$$ is a free curve with $mdr(f)=1$ and a singularity in $p=(0:0:1)$; it belongs to the unimodal family of singularity germs $W_{13}$, see \cite[Theorem II]{Arnold}. Since $\tau_p(C)=13=\mu_p(C)$, it follows that $p$ is a quasi-homogeneous singularity.

\vskip 2mm
\item A unicuspidal quartic with the tangent in the cuspidal point 
(see \cite[Case Q10 with $t=1$, page 271]{Wall}
of equation $$C=V(f)=V(x_0(x_0^3x_2+x_0^2x_1^2+x_1^4))$$ 
is a free curve with $mdr(f)=2$ and a singularity at $p=(0:0:1)$; this is another member of the unimodal family $W_{13}$. Since $\tau_p(C)=12$ and $\mu_p(C)=13$, it follows that $p$ is a non quasi-homogeneous singularity. 
\end{enumerate}
\end{example}

Free curves and nearly free curves have been deeply studied in the last decade, but, as far as we know, it is difficult to check if their singular points are quasi-homogeneous. In what follows we will establish a criterion in terms of a syzygy matrix of the Jacobian ideal.

\vskip 2mm 
A result of du Plessis and Wall gives upper and lower bounds for the total Tjurina number $\tau(C)$ of a reduced plane curve $C=V(f)\subset \PP^2$ in terms of its degree $d$ and the number ${\rm mdr}(f)$. 
It turns out that the freeness of a curve is related with $\tau(C)$. Moreover, a result of Dimca and Sticlaru 
relates also nearly free curves with $\tau(C)$. We summarize these results in the following:

\begin{proposition}\label{bounds_tau} 
Let
$C=V(f)$ be a reduced singular plane curve of degree $d$ and let $r:={\rm mdr}(f)$. Then it holds:
\begin{enumerate}
\item the global Tjurina number satisfies
\begin{equation}\label{first bound}
(d-1)(d-r-1)\le \tau(C)\le (d-1)(d-r-1)+r^2;
\end{equation}
Moreover, if $\tau(C)= (d-1)(d-r-1)+r^2$, then the curve $C$ is free, and such a condition is also sufficient if $d> 2r$.

If $d \ge 2r$, then $\tau(C)= (d-1)(d-r-1)+r^2 -1$ if and only if $C$ is a nearly free curve.
\item If, in addition, we have $2r + 1 > d$, then:
\begin{equation}\label{second bound}
\tau(C) \le (d - 1)(d - r - 1) + r^2 - {(2r + 1 - d)(2r + 2 - d)} /{2} = \frac {d(d - 1)}{2}-r^2 + r(d - 2).
\end{equation}
\end{enumerate}
\end{proposition}
\begin{proof}
See \cite[Theorem 3.2]{PW} and \cite[Corollary 1.2 and Theorem 1.3] {D}.
\end{proof}

\subsection{Liaison theory} 

In Section 5 we shall characterize non quasi-homogeneous singularities of nearly free curves and we shall use liaison theory as one of the main tools, so we shortly recall it for the sake of completeness.

\begin{definition} 
Let $X$ be a smooth irreducible projective scheme. Assume that $X$ is subcanonical (that is, $\omega _X\cong \sO_X(\lambda )$ for some $\lambda \in \ZZ$). Let $V_1, V_2\subset X$ be two equidimensional closed subschemes of codimension $c$ and assume that they do not share any common irreducible component. We will say that $V_1$ and $V_2$ are {\em directly linked} if the schematic union $Y=V_1\cup V_2$ is a complete intersection, codimension $c$, subscheme of $X$.  
\end{definition}

\begin{example}
  A simple example of schemes directly
linked is the following one: let $C_1$ be a
twisted cubic in $\PP^3$ and let $C_2$ be a secant line to $C_1$.
The union of $C_1$ and $C_2$ is a degree 4 curve which is the
complete intersection $X$ of two quadrics $Q_1$ and $Q_2$. So
$C_1$ and $C_2$ are directly linked by the complete
intersection $X$. More precisely, we take $C_1\subset \PP^3$ the
twisted cubic with homogeneous ideal
$$I(C_1)=(x_0x_2-x_1^2,x_0x_3-x_1x_2,x_1x_3-x_2^2)\subset
\mathbb{C}[x_0,x_1,x_2,x_3],$$ and the complete intersection ideal
$$I(X)=(x_1^2-x_2x_0+x_2^2-x_1x_3,x_1^2-x_0x_2+2x_2^2-2x_1x_3)\subset
I(C_1)$$ the residual to $C_1$ in $X$ is the line $C_2$ defined by
$I(C_2)=(x_1,x_2)$.
\end{example}

A useful feature of liaison
is that through the mapping cone procedure, it is possible to
pass from a free resolution of a scheme $V_1$ to a free
resolution of its residual scheme $V_2$ in a complete intersection $Y$.
Indeed, we have:

\begin{proposition}\label{resolution}
Let $X$ be a smooth irreducible subcanonical projective scheme and let $V_1, V_2\subset X$ be two equidimensional closed subschemes of codimension $c$, directly linked by a complete intersection $Y\subset X$.
Let 
$$ 
0\rightarrow \sO _{X}(-t)\rightarrow \sF _{c-1} \rightarrow \cdots \rightarrow \sF _1\rightarrow \sI _{Y,X}\rightarrow 0
$$
and
$$ 
0\rightarrow \sG _c\rightarrow \sG _{c-1} \rightarrow \cdots \rightarrow \sG _1\rightarrow \sI _{V_1,X}\rightarrow 0
$$
be locally free resolutions of $\sI _{Y,X}$ and $\sI _{V_1,X}$. Then 
they induce a locally
free resolution of $\sI _{V_2,X}$ of the following type
 $$
 0\rightarrow \sG_1^{\vee }(-t) \rightarrow \sF
_1^{\vee }(-t)\oplus \sG _2^{\vee }(-t)\rightarrow \cdots \rightarrow \sF
_{c-1}^{\vee }(-t)\oplus \sG _{c}^{\vee }(-t)\rightarrow \sI _{V_2,X}\rightarrow 0.
$$  
\end{proposition}

\begin{proof}
See \cite[Section 3]{Ha} and \cite[Proposition 2.5]{Pe}.   
\end{proof}

\section{Characterization of quasi-homogeneous singularities of free curves}

In this section we focus on free plane curves $C=V(f)$ satisfying $d > 2r$.
Using polar maps and a first syzygy matrix associated with the Jacobian ideal $J_f$ we will prove Theorem \ref{main1}. Recall that free curves $C=V(f)\subset \PP^2$ admit a minimal free resolution of the type 
$$
0 \longrightarrow R(1-d+r)\oplus R(-r) \overset {M_f}{\longrightarrow} R^3 \longrightarrow J_f (d-1)\longrightarrow 0. 
$$

In particular, the gradient $\nabla f$ of a free curve can be expressed with the order two minors of a Hilbert-Burch matrix associated with the Jacobian ideal; we shall call a Hilbert-Burch matrix $M_f$ also a {\em first syzygy matrix} associated with $J_f$:
\begin{equation}\label{eq: Hilbert-Burch}
  M_f=\begin{pmatrix} g_0 & l_0\\
g_1 & l_1\\
g_2 & l_2\\
\end{pmatrix},
\end{equation}
where $g_i \in R_r$ and $l_i \in R_{d-1-r}$.

This allows us to give a precise description of the closure $S_f$ of the graph $\Gamma_{\nabla f}$ of the polar map
$$
\nabla f : \PP^2 \dasharrow \PP^2, \quad \nabla f(p)=(\partial_0f(p):
\partial_1 f(p): \partial_2 f(p)).
$$
Indeed, for any $p=(p_0:p_1:p_2)\notin \Sigma_f$, where $\Sigma_f$ is the Jacobian scheme
of Definition \ref{jabscheme}, denote by $(g_0, g_1,g_2)\in {\rm Syz}(J_f)_r$ a triple giving a syzygy of minimal degree and by $(l_0,l_1,l_2) \in {\rm Syz}(J_f)_{d-1-r}$ a non proportional triple giving a syzygy of degree $d-1-r$. Then the coordinates of the point $\nabla f(p)$ are 
given by the order two minors of
$$
\begin{pmatrix} g_0(p) & l_0(p)\\
g_1(p) & l_1(p)\\
g_2(p) & l_2(p)\\
\end{pmatrix}.
$$
So $\nabla f(p)$ is
the intersection point of the two distinct lines 
$$
\nabla f(p): \ \left\{
\begin{array}{l}
g_0(p) y_0 +g_1(p)y_1 +g_2(p)y_2=0\\
l_0(p) y_0 +l_1(p)y_1 +l_2(p)y_2=0.
\end{array}
\right.
$$

As a consequence, the fiber of $\nabla f$ over a general point $q=(q_0:q_1:q_2)\in \PP^2$ is contained in the zero locus of
$$
 \left\{
\begin{array}{l}
g_0 q_0 +g_1 q_1 +g_2 q_2=0\\
l_0 q_0 +l_1q_1 +l_2q_2=0.
\end{array}
\right.
$$
It follows that $S_f={\overline \Gamma}_{\nabla f}$ is contained in the complete intersection $Z \subset \PP^2 \times \PP^2$ of the two divisors
given by the equations
\begin{equation}\label{eq: Z}
T=V
(g_0 y_0 +g_1 y_1 +g_2 y_2)\\
 \quad \text{ and } \quad
D=V(l_0 y_0 +l_1y_1 +l_2y_2),
\end{equation}
where $((x_0:x_1:x_2),(y_0:y_1:y_2)) \in \PP^2 \times \PP^2$.
 By choosing $L_1$ and $L_2$ as generators of the Picard groups of the two factors of $\PP^2 \times\PP^2$, and by setting $p_i:\PP^2 \times \PP^2 \to \PP^2$ to be the two projections, we have that the two divisors
$h_1=p_1^\star L_1$ and $h_2=p_2^\star L_2$
are generators for the Chow ring $A(\PP^2 \times \PP^2)$. With such a notation we have
$$
T\sim r h_1+h_2\\
 \quad \text{ and } \quad D \sim (d-1-r)h_1+h_2.
$$
In particular, the class of the complete intersection $Z_f:=T \cap D$ is 
\begin{equation}\label{eq: class of ci}
\begin{array}{rcl} Z_f & \equiv & T \cdot D= (rh_1 + h_2) \cdot ((d-1-r)h_1+h_2) \\
& = & r(d-1-r)h_1^2 + (d-1)h_1 h_2 + h_2^2 \\
& = & ((d-1)^2 - \tau (C))h_1^2 + (d-1)h_1 h_2 + h_2^2.
\end{array}
\end{equation}

On the other hand, the class of $S_f$ in the Chow ring $A(\PP^2 \times \PP^2)$ can be determined as follows.

\begin{lemma}
  The class of $S_f={\overline \Gamma}_{\nabla f}$ in $A (\PP^2 \times \PP^2)$ is given by
\begin{equation}\label{eq: class of S}
S_f\equiv (\deg \nabla f)\ h_1^2
+ (d-1)h_1 h_2 + h_2^2,
\end{equation}
where $\deg \nabla f= (d-1)^2 - \mu(C)$.
\end{lemma}

\begin{proof}
  Since $S_f$ is a codimension two cycle, its class can be written in the form
 $$
 S_f \equiv \alpha \ h_1^2 +\beta\ h_1h_2 + \gamma\ h_2^2,
 $$ 
 for some coefficients $\alpha, \beta, \gamma \in \ZZ$. Being the closure of a graph, it is birational to $\mP^2$ via the first projection $p_1$; in particular, $S_f\cdot h_1^2=1$, so
 $\gamma =1$. Moreover, the coefficient $\alpha =S_f \cdot h_2^2$
 is the cardinality of a general fiber of ${p_2}_{|S_f}$.
 By construction we have
 $$
 \rho_f:={p_2}_{|S_f}= \nabla f \circ {p_1}_{|S_f},
 $$
 and as ${p_1}_{|S_f}: S_f \to \PP^2$ is birational,
 we conclude that $\alpha$ is equal to the degree of the polar map:
 \[
 \alpha = {\rm deg} \nabla f.
 \]
 
 Finally, by the projection formula, we have 
 $$
 \beta =S_f \cdot h_1 h_2=
 S_f \cdot h_1 \cdot \rho_f^\star L_2={\rho_f}_\star (S_f \cdot h_1) \cdot L_2=
 $$
 $$
 = \nabla f_\star L_1 \cdot L_2 = d-1.
 $$
\end{proof}

\begin{corollary}
 We have that 
$$
\mu(C)=\tau(C) 
\iff Z_f \text{\ is \ irreducible.}
$$
Moreover, if $Z_f$ is reducible, then $S_f \subsetneq Z_f$ and
$$
Z_f = S_f + \sum_{i=1}^s m_i\  p_1^{-1} (P_i),
$$
where $\{P_1, \cdots , P_s\} \subseteq {\rm Sing} \ C$ and $m_i \ge 1$ are suitable multiplicities.
\end{corollary}
\begin{proof}
The first statement follows by observing that
$$
\mu(C)=\tau(C) 
 \iff \deg \nabla f = (d-1)^2 - \tau(C) \iff S_f=Z_f  \iff Z_f \text{\ is \ irreducible.}
$$
  
  The statement immediately follows by comparing the classes \eqref{eq: class of ci} and \eqref{eq: class of S}, from which is clear that the only possible irreducible components of $Z_f$ different from $S_f$ are vertical planes of class $h_1^2$. Moreover, we consider  the Jacobian scheme $\Sigma_f$ of $C$ (see Definition \ref{jabscheme}) and we observe that over $\PP^2 \setminus \Sigma_f$, the surfaces $S_f$ and $Z_f$ coincide by construction, so $Z_f = S_f \cup \bigcup_{i=1}^s p_1^\star P_i$,
with $\{P_1, \cdots , P_s\} \subseteq {\rm Sing} \ C$.
\end{proof}

It follows that checking the irreducibility of $Z_f$ amounts to checking if for some point $p=(p_0:p_1:p_2) \in \Sigma_f$, the equations
$$
g_0(p) y_0 +g_1(p) y_1 +g_2(p) y_2=0\\
 \quad \text{ and } \\ \quad
l_0(p) y_0 +l_1 (p)y_1 +l_2(p)y_2 =0
$$
vanish identically. This corresponds, in turn, to the vanishing of a Hilbert-Burch matrix \eqref{eq: Hilbert-Burch} in the point $p$. In other words, the irreducibility of $Z_f$ corresponds to the emptiness of the rank zero locus
of $M_f$:
$$
Y_f:= \{p\in \Sigma_f \ | \ {\rm rk} \ M_f (p)=0\}
\subseteq \Sigma_f.
$$
Observe that $Y_f$ can be given the natural structure of scheme induced by the homogeneous ideal
$$
\langle g_0,g_1,g_2,l_0,l_1,l_2\rangle.
$$
We summarize these observations in the following:

\begin{theorem}\label{main1}
  If $C=V(f)\subset \PP^2$ is a free curve, then 
$$
\mu (C)=\tau(C) \iff Y_f = \emptyset.
$$
  
\end{theorem}

\begin{proposition}
  Let $C=V(f)\subset \PP^2$ be a free curve with ${\rm mdr} (f)=r$, and let $q$ be the number of non quasi-homogeneous singular points. Then 
   \begin{equation}
       \label{eq: bound on number of non qh points}
       q \le r^2.
   \end{equation}
    
\end{proposition}
\begin{proof}
   As $q$ is the cardinality of the finite set $Y_f$, and since $Y_f=V(g_0,g_1,g_2,l_0,l_1,l_2)$, where the $g_i$ and the $l_j$ are the entries of the matrix $M_f$, we see that $Y_f \subseteq V(g_0,g_1,g_2)$. So, we have $q \le r^2$.
\end{proof}

\begin{example}\label{example: Ploski curves} 
Consider the {\it P\l oski curves} (see \cite[Definition 1.7]{C}, \cite[Theorem 1.1]{D2}), that is curves given by unions of conics belonging to a hyperosculating pencil in even degree, respectively unions of conics belonging to a hyperosculating pencil union the tangent line through the hyperosculation point in odd degree. Such arrangements have only one singular point $p$, which is non quasi-homogeneous, and it is referred as the {\it second worst singularity} in \cite{C}. Such curves attain the upper bound
$\mu (p) \le (d-1)^2-\left\lfloor\frac{d}{2}\right\rfloor$ for the local Milnor number, and it holds
$$
\mu (p) = (d-1)^2-\left\lfloor\frac{d}{2}\right\rfloor> \tau(p) = d^2-3d+3=(d-1)^2 - (d-2),
$$
see \cite{Pl}.

Such curves are free, see for instance \cite{D2} and \cite[Example 3.2(2)]{BMR}, with exponents $(1,d-2)$. Let us determine explicitly a Hilbert-Burch matrix of the Jacobian ideal.

  Without loss of generality, we can assume that $p=(0:0:1)$. In the even degree case the equation of the curve is given by

$$
{\mathcal C}_1:f=\prod _{i=1}^m(x_0^2 + a_i (x_0x_2 + x_1^2)) = 0, 
$$
 with $a_i\ne 0$ and $a_i\ne a_j$ for $i\ne j$.
 It is simple to check that
 ${\rm Syz}(J_f)_1$ is generated by $(0,x_0,-2x_1)$.
To determine a syzygy in ${\rm Syz}(J_f)_{d-2}$ first note that $\partial_0 f\in (x_0,x_1)$, thus $\partial_0 f= x_0g+x_1h$ with $g,h\in R_{d-2}$. It follows that $\frac{\partial_1 f}{2x_1}=\frac{\partial_2 f}{x_0}$ is an element of $R$ and $\left(\frac{\partial_1 f}{2x_1}=\frac{\partial_2 f}{x_0},\frac{h}{2},g\right) \in {\rm Syz}(J_f)_{d-2}$. Therefore a Hilbert-Burch matrix of $J_f$ is given by  
\begin{equation}\label{eq: Hilbert Burch C1}
M_f=
\begin{pmatrix} 0 & \frac{\partial_1 f}{2x_1} \\
x_0 & \frac {h}{2}\\
-2x_1 & g\\
\end{pmatrix}
\end{equation}
and $M_f$ is annihilated by the point $p=(0:0:1)$, which confirms that the point $p$ is non quasi-homogeneous.

In the odd degree case, the equation has the form
$$
{\mathcal C}{\mathcal L}_1: \ f=x_0\prod _{i=1}^m(x_0^2 + a_i (x_0x_2 + x_1^2)) = 0,
$$
with $a_i\ne 0$ and $a_i\ne a_j$ for $i\ne j$.

 The linear part of ${\rm Syz}(J_f)_1$ is generated by $(0,x_0,-2x_1)$ also in this case. 
 
 Like in the previous case we have $\left(\frac{\partial_1 f}{2x_1}=\frac{\partial_2 f}{x_0},\frac{h}{2},g\right) \in {\rm Syz}(J_f)_{d-2}$. Therefore a Hilbert-Burch matrix of $J_f$ is given again by \eqref{eq: Hilbert Burch C1} and it is annihilated by the point $p=(0:0:1)$, which confirms that this point is non quasi-homogeneous.

\end{example}

Although the previous examples are free conic-line arrangements, it is important to highlight that the criteria work for any reduced free curve. 

\begin{example}
 Consider the rational cuspidal curves described in
 \cite[Theorem 1.1]{DS4}:
$$C=V(f)=V(x_0^d+x_0^rx_1^{d-r}+x_1^{d-1}x_2),$$
\noindent with $2\leq r <d/2$. These curves are free with exponents $(r,d-r-1)$, one unique cusp at $p=(0:0:1)$ and it can be checked that a Hilbert-Burch matrix is
\begin{equation}
M_f=
\begin{pmatrix} 
0& (d-r)^2x_1^{d-r-1}\\
x_1^r& -d((d-r)x_0^{d-r-1}-(d-1)x_0^{d-1-2r}x_1^{r-1})\\
(d-1)x_1^{r-1}x_2& -(r(d-r)^2x_0^{r-1}x_1^{d-2r}+d(d-1)^2x_0^{d-1-2r}x_1^{r-2}x_2^ 2)
\end{pmatrix}.
\end{equation}
So $M_f (0:0:1)= 0$ for all $d\geq 6$ and $2\leq r<d/2$ and the point is non quasi-homogeneous.

\end{example}

\begin{remark}
  We point out that the bound \eqref{eq: bound on number of non qh points} is solely set-theoretic, and it does not give any insight in the difference $\mu_p(C) - \tau_p(C)$ in a non quasi-homogeneous point $p$. This is clear in the examples of P\l oski curves of Example \ref{example: Ploski curves}. 
    
    Indeed, in both cases we have $r=1$, and a unique non quasi-homogeneous point, hence $q=r^2$. However, if $d \ge 7$, it holds
    $$
    \mu(C) - \tau(C)= (d-1)^2 - \left\lfloor \frac {d}{2}\right\rfloor - ((d-1)^2-(d-2)))=d-2-\left\lfloor \frac {d}{2}\right\rfloor \ge 2>r^2.
    $$
\end{remark}

\section{On the locus of Hilbert-Burch matrices of free curves with only quasi-homogeneous singularities}

Now we would like to give some insight in the locus of Hilbert-Burch matrices of free plane curves having only quasi-homogeneous singularities.

We shall identify the space of all $3 \times 2$ matrices with a column given by three homogeneous forms of degree $r \ge 1$ and a column given by three homogeneous forms of degree $d-1-r$, with the product
$$
\mathcal {M}_{d,r} :=\mP(R_r ^{\oplus 3}) \times \mP (R_{d-1-r}^{\oplus 3}).
$$
Then we can consider the locus of Jacobian Hilbert-Burch matrices, and we get the following result:

\begin{proposition}\label{prop: Jacobian Hilbert Burch locus}
    Let $\mathcal {J}_{d,r}\subset \mathcal {M}_{d,r}$ be the locus of Hilbert-Burch matrices of Jacobian schemes of free curves. Then the closure of $\mathcal {J}_{d,r}$ is defined by quadratic equations.

    Moreover, the locus of Hilbert-Burch matrices corresponding to free plane curves with all quasi-homogeneous singularities is open and non-empty for any $d \ge 2$ and any $r < \frac{d}{2}$.
\end{proposition}

The proof will rely on the exactness of the de Rham complex on $\mathbb{A}^3$, which we briefly recall; we follow the notation of
\cite[Definition 3.3]{BGV}.

\begin{def}
\label{def: de Rham complex}
Let $R = \mathbb{C}[x_0,x_1,x_2]$. The \emph{de Rham complex} of $\mathbb{A}^{3}$ is the complex of $R$-modules defined as
	\[
	0\to \Omega^0_{\mathbb{A}^{3}}\to\Omega^1_{\mathbb{A}^{3}}\to\Omega^2_{\mathbb{A}^{3}}\to\Omega^{3}_{\mathbb{A}^{3}}\to 0,
	\]
	where $\Omega^k_{\mathbb{A}^{3}}$ is the module of algebraic differential $k$-forms over $\mathbb{A}^{3}$. By letting $dx_0,dx_1,dx_2$ be a $R$-basis for $\Omega^1_{\mathbb{A}^{3}}$ we can write a $k$-form $\omega = \sum_{I\in\binom{[3]}{k}}f_I dx_I\in \Omega^k_{\mathbb{A}^{3}}$, with $f_I\in R$ and with $dx_I=dx_{i_1}\wedge\cdots\wedge dx_{i_k}$ for $I = \{i_1,\cdots,i_k\}$. The differentials of the de Rham complex are given by exterior derivatives, i.e.
	\[
		d\omega = \sum_{I\in\binom{[3]}{k}}\sum_{j=0}^{2}\partial_j f_I \ dx_j\wedge dx_I.
	\]\end{def}
We are interested in the differential 
	\begin{align*}
d:\Omega^1_{\mathbb{A}^{3}}&\to\Omega^2_{\mathbb{A}^{3}}\\
		\sum_{i=0}^2 f_i \ dx_i&\mapsto \sum_{i=0}^2\sum_{j=0}^2 \partial_j f_i\ dx_j\wedge dx_i = \sum_{0\leq i<j\leq 2} (\partial_if_j-\partial_jf_i)\ dx_i\wedge dx_j.
	\end{align*}
	As the de Rham complex of $\mathbb{A}^{3}$ is exact everywhere but in degree 0 (this follows, for instance, from de Rham's theorem relating the cohomology of de Rham complex with the singular cohomology of $\mathbb{A}^{3}$), we conclude the following.
 
 \begin{lemma}   \label{lemma: de rham}
 A triple of polynomials $(f_0,f_1,f_2)$ satisfies the relations 
 $$
 \partial_if_j-\partial_jf_i=0, \ \forall 0\leq i<j\leq 2,
 $$
 if and only if there exists an element $g \in \Omega^0_{\mathbb{A}^{3}}\cong R$ such that 
 $$
 \nabla g = (f_0,f_1,f_2).
 $$
 
 If moreover the polynomials $f_0,f_1,f_2$ are homogeneous of degree $d-1$, then by Euler's formula $g$ is homogeneous of degree $d$.
	\end{lemma} 

We are now in the position of proving Proposition \ref{prop: Jacobian Hilbert Burch locus}.

\noindent{\it Proof of Proposition \ref{prop: Jacobian Hilbert Burch locus}}
    Given $M \in \mathcal {M}_{d,r}$ with
    $M= \begin{pmatrix} g_0 & l_0\\
g_1 & l_1\\
g_2 & l_2\\
\end{pmatrix}$, we set
$M_i$ to be the determinant of the minor obtained by deleting the $(i+1)$-th row.
We have that $M = M_f$ for some free plane curve $V(f)$ if and only if the triple $(M_{0}: - M_{1}: M_{2} )$ is proportional to 
 $(\partial_0 f: \partial_1 f: \partial_2 f)$. Moreover, by Lemma \ref{lemma: de rham}, this holds if and only if the curl of $(M_{0}: - M_{1}: M_{2} )$ is identically zero:
$$
 \partial_0 M_{1} + \partial_1 M_{0}=0, \quad
 \partial_0 M_{2} - \partial _2 M_{0} =0, \quad
 \partial _1 M_{2} + \partial_2 M_{1}=0.
 $$
Such expressions are homogeneous polynomials of degree $d-2$ and they give the zero polynomials if and only if the coefficients of all monomials
in $x_0,x_1,x_2$ are zero. These coefficients are, in turn, quadratic polynomial in the coefficients of the entries $g_i$ and $l_j$ of $M$.

Next, we focus on the elements of $\mathcal {J}_{d,r}$ which have never rank zero.
Recall that the general matrix in $\mathcal {M}_{d,r}$ has rank $\ge 1$ everywhere. Indeed, the locus
$$
\mathcal {N} \subset {\mathcal M}_{d,r} =\mP(R_r ^{\oplus 3}) \times \mP (R_{d-1-r}^{\oplus 3}), 
$$
$$
\mathcal N:=\{ ([g_0],[g_1],[g_2],[l_0],[l_1],[l_2]) \ | \ \langle g_0,g_1,g_2,l_0,l_1,l_2\rangle_{d-1-r} \subsetneq R_{d-1-r} \}
$$
is closed, as it is the zero locus of the order ${d+1-r} \choose {2}$ minors of the coefficient matrix of the generators $g_0,g_1,g_2,l_0,l_1,l_2$ in the standard monomial basis of $R_{d-1-r}$.

It remains to be shown that the complementary open locus is non-empty, as provided by the next example.
$\square$
\begin{example}\label{eq: ex of free with only qh}
 It is known that all the singularities of line arrangements 
are quasi-homogeneous, see for instance \cite{Schenck1}. Thus, examples of free curves with $d\geq 2$, $r\geq 1$ and a never-zero Hilbert-Burch matrix are given by 
the line arrangements described in
 \cite[Theorem 1.2]{DS4}:
 \begin{equation}
 {C}=V(f)=V(x_0\ g(x_0,x_1)\ h(x_0,x_2)),
 \end{equation}
 where $g$ is homogeneous, square-free, of degree $r$ with $\frac{d}{2}> r \ge 1$ and $h$ is homogeneous and square-free of degree $d-1-r$. Assume, moreover, that 
 $$
 g(x_0,x_1)= \Pi _{i=1}^{r} (a_i x_0 + b_i x_1),
 $$
 with $b_i \neq 0$ for any $i =1, \cdots , r$ and
 $(a_i:b_i)\ne (a_j:b_j)$ for $i\neq j$, and that
 $$
 h(x_0,x_2)= \Pi _{i=1}^{d-1-r} (c_i x_0 + d_i x_2),
 $$
 with $d_i \neq 0$ for any $i =1, \cdots , d-1-r$ and
 $(c_i:d_i)\ne (c_j:d_j)$ for $i\neq j$.

 \end{example}

\begin{remark}
  There is rational map from $\varphi:\mathcal {M}_{d,r} \dasharrow \PP(R_d)$
 defined outside the locus of matrices with proportional columns, and given by $\varphi (M) = 
 x_0 M_0 - x_1 M_1 +x_2 M_2$.
 It maps
 the locally closed locus 
 $\mathcal {J}_{d,r}$ to the locus of free plane curves $C=V(f)$ of degree $d$ with ${\rm mdr}(f)=r$. The geometric properties of such a map are not evident and deserve further investigation.
\end{remark}

\begin{remark} We would like to emphasize that there are also examples of curves with a never-zero Jacobian Hilbert-Burch matrix that are not line arrangements. 

For instance, one can consider the curve $C=V(f)=V((x_0^d+x_1^d+x_2^d)(x_0^d+x_1^d))$, where $d\geq 2$; according to \cite[Corollary 1.6]{DIPS}, the curve $C$ is free of degree $2d$, with exponents $(d-1,d)$. In this case, a Hilbert-Burch matrix is given by
\begin{equation}
M_f=
\frac{1}{d}
\begin{pmatrix}  -x_1^{d-1}  & -x_0x_2^{d-1} \\
 x_0^{d-1}  & -x_1x_2^{d-1} \\
0 & 2x_0^{d}+2x_1^d+x_2^d\\
\end{pmatrix}.
\end{equation}
Since $\text {rk}~
M_f (p)= 0$ if and only if $x_0=x_1=x_2=0$ it follows that 
$\text {rk}~
M_f (p)\neq 0$ for all $d+1$ singular points of $C$. Therefore all the singularities are quasi-homogeneous.

\end{remark}

\subsection{Free curves in low degrees}
  We now restrict our attention to free plane curves $C=V(f)\subset \PP^2$ of low degrees and we will compare our results with some known facts about such curves. Recall that the exponents $(r,d_2)$ satisfy $r+d_2=d-1$, so it is always $r < \frac{d}{2}$.

  If a free curve has degree $d=3$ and it is not a cone, then $1\le r < \frac{3}{2}$ and we necessarily have $r=1$. Moreover, by
  Proposition \ref{bounds_tau},
  a curve is free if and only if $\tau(C)= (d-1)(d-2)+1=3$. Since $\tau(C) \le \mu(C) <(d-1)^2=4$,
 in this case we always have $\tau(C)=\mu(C)$, that is all the singularities are quasi homogeneous.
 Moreover, the degree of the polar map is $1$, that is $f$ is a homaloidal polynomial; the latter have been classified by Dolgachev in \cite{Dol}, and can be either the union of an irreducible conic and a tangent line or the union of three general lines. 

 A degree $d=4$ free curve satisfies $1 \le r < \frac{4}{2}$,
 so $r=1$ also in this case, and $\tau(C)= (d-1)(d-2)+1=7$. By the Dolgachev characterization of homaloidal polynomials, we have that
 ${\rm deg} \nabla f = (d-1)^2 - \mu (C) =9 -\mu(C)\ge 2$. Therefore, 
 $\mu(C) \le 7$ and again all the singularities of free curves are quasi-homogeneous. The quartic curves with polar map of degree $2$ have been classified in \cite[Theorem 3.3]{FM}, and they are the following
 \begin{itemize}
  \item three concurrent lines union a general line (the line arrangement $\sL$ of \cite{BMR});
  \item a smooth conic, a tangent line and a line passing through the tangency point;
  \item a smooth conic and two tangent lines;
  \item two smooth hyperosculating conics (a degree $4$ P\l oski curve of type $\sC_1$ of \cite{BMR});
  \item an irreducible cuspidal cubic and its tangent at the smooth flex point (see \cite[Theorem 3.3, type 8]{FM});
  \item an irreducible cuspidal cubic and its tangent at the cusp (see \cite[Theorem 3.3, type 9]{FM}).
 \end{itemize}

 A degree $d=5$ free curve satisfies $1 \le r < \frac{5}{2}$, so $r=1$ or $r=2$. If $r=1$, we have $\tau(C)= (d-1)(d-2)+1=13$ and if $r=2$ we have $\tau(C)=(d-1)(d-3)+4=12$. Moreover,
 by the Dolgachev classification again, we have
 $2 \le {\rm deg} \nabla f = (d-1)^2 - \mu (C) =16 -\mu(C)$, so 
 $\mu(C) \le 14$. 
 
 If $r=1$, we have $13=\tau(C) \le \mu(C) \le 14$, so the degree of the polar map is $2 \le {\rm deg} \nabla f \le 3$. By the classification given in the \cite[Theorem 3.3 and Theorem 3.4]{FM}, by \cite{BMR} and by the classification of highly singular quintic curves provided by Wall in \cite{Wall}, we have the following families with all quasi-homogeneous singularities:
 \begin{itemize}
     \item four lines concurrent in a point $p$ union a general line, the line arrangement $\sL$ of \cite{BMR} (see \cite[Theorem 3.4, type 1]{FM} and type H11 in \cite{Wall}). 
    \item the arrangement $\sC \sL_5$ of \cite{BMR} (see \cite[Theorem 3.4, type 5]{FM} and type H4 in \cite{Wall}). 
    \item the arrangement $\sC \sL_2$ of \cite{BMR} (see \cite[Theorem 3.4, type 5]{FM} and type H12 in \cite{Wall}). 
    \item an irreducible cuspidal cubic union the tangent at the cusp and the tangent at the flex point (see \cite[Theorem 3.4, type 15]{FM} and type H1 in \cite{Wall}).
    \item an irreducible cuspidal cubic union the tangent at the cusp and a secant line passing through the cusp and a smooth point (different of the flex point) (see \cite[Theorem 3.4, type 16]{FM} and type H10 in \cite{Wall}).
    \item an irreducible cuspidal cubic union the tangent at the flex point and a secant line passing through the cusp and the flex point (see \cite[Theorem 3.4, type 17]{FM} and type H3 in \cite{Wall}).
    \item an irreducible unicuspidal oval quartic union the tangent at the cusp (see \cite[Theorem 3.4, type 28]{FM} and type H9 in \cite{Wall}).
    \item an irreducible unicuspidal oval quartic union a tangent at the flex point (see \cite[Theorem 3.4, type 29]{FM}) and type H2 in \cite{Wall}.
    \end{itemize}
  \vspace{6pt}
  
The only family, in which a non quasi-homogeneous singularity appears, is given by the arrangement $\sC \sL_1$ of \cite{BMR} (see \cite[Theorem 3.3, type 6]{FM} and type H8 in \cite{Wall}).

\vspace{6pt}

  If $r=2$, we have $12=\tau(f) \le \mu(f) \le 14$, so the degree of the polar map is $2\le {\rm deg} \nabla f \le 4$. According to the classification of \cite{FM} there is no free curve of degree $5$ with $r=2$ and ${\rm deg} \nabla f=2$. For $3\leq{\rm deg} \nabla f\leq 4 $ there are many cases; however, we can provide examples of curves with quasi-homogeneous singularities and examples of curves with non quasi-homogeneous singularities:

 \begin{itemize}
  
\item The curve of type Q7 in \cite{Wall} given by $f=x_0^3x_1x_2+x_0^5+x_0x_1^4$ has a quasi-homogeneous singularity at $(0:0:1)$, and it belongs to the unimodal family
$Z_{12}$.
\item The curve of type Q10 in \cite{Wall} given by $f=x_0^4x_2+tx_0^3x_1^2+x_0x_1^4$ ($t\neq 0$) has a non quasi-homogeneous singularity at $(0:0:1)$, and it belongs to the unimodal family $W_{13}$. It is important to highlight that for $t=0$, this curve is also free but it has exponents $(1,3)$ and, as noted previously, it is an irreducible unicuspidal quartic union a tangent, with only quasi-homogeneous singularities (see Example \ref{singW13}).

 \end{itemize}

\section{Characterization of quasi-homogeneous singularities of nearly free curves}\label{nonqh}

In this section, we consider nearly free plane curves $C=V(f)$ of degree $d$ satisfying $d \ge 2r$. We shall use again polar maps and a first syzygy matrix $M_f$ of the Jacobian ideal $J_f$ of $f$ to prove Theorem \ref{main2}. 

By \eqref{nearly free}
we have an exact sequence of the type 
\begin{equation}\label{eq: nearly free sequence}
0 \longrightarrow R(-d+r-1) 
\overset {P_f}{\longrightarrow} R(-r)\oplus R(-d+r)^2 \overset {M_f}{\longrightarrow} R^3 \longrightarrow J_f (d-1)\longrightarrow 0
\end{equation}
with $r\le \frac{d}{2}$.
We will denote the entries of the {\em first syzygy matrix} $M_f$ of $J_f$ as follows:
\begin{equation}\label{eq: nearly free first syzygy matrix}
 M_f=\begin{pmatrix}
    A_0 & B_0 & C_0\\
    A_1 & B_1 & C_1 \\
    A_2 & B_2 & C_2 \\
  \end{pmatrix},
\end{equation}
where $A_i \in R_r$ and $B_j, C_k \in R_{d-r}$, and the entries of the second syzygy matrix $P_f$ with
\begin{equation}\label{eq: nearly free second syzygy matrix}
  P_f=\begin{pmatrix}
    P_0 \\
    P_1 \\
    P_2 \\
  \end{pmatrix},
\end{equation}
where $P_0 \in R_{d-2r+1}$ and $P_1,P_2 \in R_1$.
\begin{lemma}
Let $C=V(f)\subset \PP^2$ be a nearly free curve with first syzygy matrix \eqref{eq: nearly free first syzygy matrix}. Then it holds
\begin{equation}
  \det \begin{pmatrix}
    x_0 & B_0 & C_0\\
    x_1 & B_1 & C_1 \\
    x_2 & B_2 & C_2 \\
  \end{pmatrix}= d \cdot \alpha(x_0,x_1,x_2) \cdot f,
  \quad \det \begin{pmatrix}
    A_0 & x_0 & C_0\\
    A_1 & x_1 & C_1 \\
    A_2 & x_2 & C_2 \\
  \end{pmatrix}= d \cdot \beta(x_0,x_1,x_2) \cdot f,
  \end{equation}
  \begin{equation}
 \det \begin{pmatrix}
    A_0 & B_0 & x_0\\
    A_1 & B_1 & x_1 \\
    A_2 & B_2 & x_2 \\
  \end{pmatrix}= d \cdot \gamma(x_0,x_1,x_2) \cdot f,
\end{equation}
for suitable polynomials 
$$
\alpha \in R_{d-2r+1} \quad \text{ and } \quad \beta, \gamma \in R_{1}.
$$
\end{lemma}
\begin{proof}
From the exactness of the sequence we have the following relations:
\begin{equation}\label{eq: equations of nearly free syzygies }
\left\{
  \begin{array}{ll}
  A_0 \partial _0 f +A_1 \partial_1 f + A_2 \partial_2 f&=0\\
  B_0 \partial _0 f +B_1 \partial_1 f + B_2 \partial_2 f&=0\\
  C_0 \partial _0 f +C_1 \partial_1 f + C_2 \partial_2 f&=0\\
  \end{array}.
  \right .
\end{equation}
In particular, we see that $\nabla f$ is proportional to the order two minors of any pair of columns of \eqref{eq: nearly free first syzygy matrix}, so by Euler formula, we get the statement.
\end{proof}

Next, we would like to describe the graph of the polar map $\nabla f:\PP^2 \dasharrow \PP^2$. Like in the free case, we observe that given a point
$p \in \PP^2 \setminus \Sigma_f$ and $q\in \PP^2$, if $q=(q_0:q_1:q_2)=\nabla f(p)$, then the following equations are satisfied:
\begin{equation}
\left\{
  \begin{array}{ll}
  A_0 (p) q_0 +A_1 (p)q_1 + A_2 (p)q_2&=0\\
  B_0(p) q_0+B_1 (p)q_1+ B_2 (p)q_2&=0\\
  C_0(p) q_0 +C_1 (p)q_1 + C_2(p) q_2&=0\\
  \end{array}.
  \right.
\end{equation}
As a consequence the closure
$S_f = {\overline \Gamma}_{\nabla f}\subset \PP^2 \times \PP^2$ is contained in the locus
\begin{equation}
\left\{
  \begin{array}{ll}
  A_0 y_0 +A_1 y_1 + A_2 y_2&=0\\
  B_0 y_0+B_1 y_1+ B_2 y_2&=0\\
  C_0 y_0 +C_1 y_1 + C_2 y_2&=0\\
  \end{array}, 
  \right .
\end{equation}
that is in the intersection of three divisors 
$$
D_A \sim rh_1 + h_2, \ D_B \sim (d-r)h_1 + h_2, \ D_C \sim (d-r)h_1 + h_2. 
$$
The three divisors determine a quasi-complete intersection surface $Z_f$ in
$\PP^2 \times \PP^2$. To determine its class we need to compute a resolution of
the ideal sheaf $\sI_{Z_f, \PP^2 \times \PP^2}$.

\begin{lemma}\label{lemma: sections of E}
 The quasi-complete intersection surface $Z_f\subset \PP^2 \times \PP^2$ 
 admits a resolution of the type
\begin{equation}\label{eq: sequence defining E}
0\to \sE \to \sO_{\PP^2 \times \PP^2} (-D_A) \oplus 
\sO_{\PP^2 \times \PP^2} (-D_B) \oplus \sO_{\PP^2 \times \PP^2} (-D_C) \to \sI_{Z_f, \PP^2 \times \PP^2} \to 0,
\end{equation}
where $\sE$ is a rank $2$ syzygy reflexive sheaf on $\PP^2 \times \PP^2$, and we have
$$
h^0 (\sE ((d-r+1)h_1+h_2))\neq 0.
$$

 Moreover, such a twist is minimal, in the sense that for any divisor 
 $ah_1 +b h_2$ such that $h^0 (\sE (ah_1+bh_2))\neq 0$, we have
 $$
 a\ge d-r+1, \qquad \text{and} \qquad b \ge 1.
 $$
\end{lemma}
\begin{proof}
 Observe that the exactness of
\eqref{eq: nearly free sequence} implies that
$$
M_f \cdot P_f = 0.
$$
As a consequence, we have also
\begin{equation}\label{eq: syzygy in p^2xp^2}
P_0 (A_0 y_0 +A_1 y_1 + A_2 y_2)
+P_1(B_0 y_0+B_1 y_1+ B_2 y_2)+P_2
 (C_0 y_0 +C_1 y_1 + C_2 y_2) \equiv 0.
\end{equation}
This gives a non-zero section of $\sE ((d-r+1)h_1+h_2)$. 
\end{proof}

In the next lemma, we will prove that $\sE$ is a rank $2$ vector bundle on $\PP^2\times \PP^2$ that splits into a direct sum of line bundles. This fact will allow us to compute the class of $Z_f$ in the Chow ring $ A(\PP^2 \times \PP^2)$.
\begin{lemma}\label{E_splits}
 With the above notation, we have:
 $$\sE=\sO_{\PP^2 \times \PP^2}(-(d-r+1)h_1-h_2)\oplus \sO_{\PP^2 \times \PP^2}(-(d-1)h_1-2h_2) .$$
\end{lemma}

\begin{proof}
 As a main tool, we will use liaison theory. We consider in $\PP^2\times \PP^2$ the surface $Y$ complete intersection of the divisors $D_A \equiv rh_1+h_2$ and $D_B\equiv (d-r)h_1+h_2$. We define $S_2\subset \PP^2\times \PP^2$ to be the surface linked to $Z_f$ via $Y$. Let us compute the resolution of $S_2$. To this end, we apply the mapping cone process (see Proposition \ref{resolution}) to the resolution of $\sI_{Z_f, \PP^2 \times \PP^2}$ and $ \sI_{Y, \PP^2 \times \PP^2} $:
 $$
0\to \sO_{\PP^2 \times \PP^2} (-dh_1-2h_2) \to \sO_{\PP^2 \times \PP^2} (-rh_1-h_2) \oplus \sO_{\PP^2 \times \PP^2} (-(d-r)h_1-h_2) \to \sI_{Y, \PP^2 \times \PP^2} \to 0,
$$ and 
 $$
0\to \sE \to \sO_{\PP^2 \times \PP^2} (-rh_1-h_2) \oplus \sO_{\PP^2 \times \PP^2} ((r-d)h_1-h_2)^2 \to \sI_{Z_f, \PP^2 \times \PP^2} \to 0,
$$
 and we get
\begin{equation}\label{eq: linked}
0 \to \sO _{\PP^2 \times \PP^2}(-rh_1 -h_2) \to \sE^\vee (-dh_1-2h_2) \to \sI_{S_2, \PP^2 \times \PP^2}\to 0.
\end{equation}
Now we recall (see \cite[Proposition 1.10]{H}) that for rank two reflexive sheaves we have the relation
$$
\sE^\vee \cong \sE (-c_1(\sE)).
$$
Moreover, we can compute the first Chern 
class of $\sE$ from \eqref{eq: sequence defining E}, and we get
$$
c_1(\sE)\sim (r-2d)h_1 - 3h_2.
$$
Hence the sequence \eqref{eq: linked} becomes
\begin{equation}\label{eq: linked bis}
0 \to \sO _{\PP^2 \times \PP^2}(-rh_1 -h_2) \to \sE ((d-r)h_1+h_2) \to \sI_{S_2, \PP^2 \times \PP^2}\to 0.
\end{equation}
As $h^0(\sE ((d-r+1)h_1+h_2))\neq 0 $, we see that 
$h^0 (\sI_{S_2, \PP^2 \times \PP^2}(h_1))\neq 0$. Thus $S_2$ is contained in the complete intersection of two divisors of the type
$h_1$ and $rh_1+h_2$. In particular, we have that the class of $S_2$ is of the type
$$
S_2 \equiv \alpha h_1^2 + \beta h_1 h_2,
$$
for some $0\le \alpha \le r$ and $0\le \beta \le 1$. Now observe that the class of $S_2$ as a divisor in the threefold $W$, where $W\sim h_1 = L_1 \times \PP^2$ for some line $L_1 \subset \PP^2$ is
$$
S_2 \sim \alpha {h_1}_{|W} +\beta {h_2}_{|W},
$$
so $S_2$ is a complete intersection of the divisor $W$ and some other divisor of class $\alpha h_1 + \beta h_2$. This implies that $\sE$ is in fact a vector bundle and it splits into a direct sum of line bundles.

Taking into account that by Lemma \ref{lemma: sections of E} the divisor $(d-r+1)h_1 + h_2$ gives a minimal twist of $\sE$ giving a section,
and taking into account that $c_1(\sE)= (r-2d)h_1 - 3h_2$ we get the statement.
\end{proof}

We are now ready to determine the class of $Z_f$. Indeed, we have:
\begin{lemma} Keeping the above notation, we have
 $$
 Z_f \equiv (r(d-1-r)+1) \ h_1^2 +(d-1) h_1 \ h_2 + h_2^2.
 $$ 
\end{lemma}
\begin{proof} 
  Since $Z_f$ is a codimension two cycle, its class can be written in the form
 $$
 Z_f \equiv \alpha \ h_1^2 +\beta h_1 \ h_2 + \gamma h_2^2,
 $$ 
 for some coefficients $\alpha, \beta, \gamma \in \ZZ$. Using the exact sequence
 $$
0\longrightarrow 
 \begin{array}{c} \sO_{\PP^2 \times \PP^2}(-(d-r+1)h_1-h_2) \\ \oplus \\ \sO_{\PP^2 \times \PP^2}(-(d-1)h_1-2h_2) \end{array} \longrightarrow 
 \begin{array}{c}
 \sO_{\PP^2 \times \PP^2} (-rh_1 - h_2) \\ \oplus \\
\sO_{\PP^2 \times \PP^2}( (-d+r)h_1 - h_2 )^2 
\end{array}\longrightarrow \sI_{Z_f, \PP^2 \times \PP^2} \longrightarrow 0
$$
and the multiplicative character of the Chern polynomial we get:
 $$
 \alpha \ h_1^2 +\beta h_1 \ h_2 + \gamma h_2^2 + ((r-d-1)h_1-h_2)((1-d)h_1-2h_2) =
 $$ 
 $$
 2 (rh_1 + h_2)( (d-r)h_1 + h_2 )+ ( (d-r)h_1 + h_2 )^2
 $$
 or, equivalently,
 $$
 (\alpha +(d-1)(d+1-r)) \ h_1^2 +(\beta +2d-r) h_1 \ h_2 + (\gamma +2) h_2^2 =
 (d^2-r^2)h_1^2+(4d-2r)h_1h_2+3h_2^2.$$
Therefore, we have $\alpha =r(d-1-r)+1 $, $\beta =d-1 $ and $\gamma =1$.
\end{proof}

We are now in the position to compare the classes of $Z_f$ and of $S_f ={\overline \Gamma} _{\nabla f} \equiv \deg \nabla f h_1^2 + (d-1) h_1 h_2 + h_2^2$. Surprisingly, we get a very similar result as in the free case.
\begin{corollary}
 Let $C=V(f)$ be a nearly free plane curve. Then 
$$
\mu(C)=\tau(C) 
\iff Z_f \text{\ is \ irreducible.}
$$
Moreover, if $Z_f$ is reducible, then $S_f \subsetneq Z_f$ and
$$
Z_f = S_f + \sum_{i=1}^s m_i p_1^{-1} (P_i),
$$
where $\{P_1, \cdots , P_s\} \subseteq {\rm Sing} \ C$ and $m_i \ge 1$ are suitable multiplicities.
\end{corollary}
\begin{proof}
The first statement follows by observing that the coefficient $r(d-1-r)+1 $ of $h_1^2$ in the class of $Z_f$ is equal to
$$
r(d-1-r)+1 = (d-1)^2 - ((d-1)(d-1-r)+r^2-1)= (d-1)^2 - \tau(C).
$$
It follows that
$$
\mu(C)=\tau(C) 
 \iff \deg \nabla f = (d-1)^2 - \tau(C) \iff S_f=Z_f \iff Z_f \text{\ is \ irreducible.}
$$
 We see that the only possible irreducible components of $Z_f$ different from $S_f$ are vertical planes of class $h_1^2$. Moreover, we observe that over $\PP^2 \setminus \Sigma_f$, the surfaces $S_f$ and $Z_f$ coincide by construction, so $Z_f = S_f \cup \bigcup_{i=1}^s p_1^\star P_i$,
with $\{P_1, \cdots , P_s\} \subseteq {\rm Sing} \ C$.
\end{proof}

It follows that also in the nearly free case checking the quasi homogeneousity amounts to checking if for some $p=(p_0:p_1:p_2) \in \Sigma_f$, the three equations
$$
M_f (p) \cdot 
\begin{pmatrix}
 y_0\\
 y_1\\
 y_2
\end{pmatrix}
$$
vanish identically; this corresponds, in turn, to the vanishing of a first syzygy matrix $M_f$ in the point $p$. Summarizing we have:

\begin{theorem} \label{main2}
 If $C=V(f)\subset \PP^2$ is a nearly free curve, then 
$$
\mu (C)=\tau(C) \iff \text {rk} \ M_f(p) \ge 1, \forall \ p \in \PP^2.
$$
 
\end{theorem}

\begin{remark}
 Observe that for any point $p \not \in \Sigma_f$, the triple $\nabla f(p)$ is a non-trivial solution of the linear system \eqref{eq: equations of nearly free syzygies }, so the matrix $M_f$ has generic rank $\le 2$. Moreover, since the three syzygies are linearly independent, generically the rank is also $\ge 2$. Hence for a general $p \in \mP^2 \setminus \Sigma_f$, we have $\text{rk} \ M_f(p)=2$.
\end{remark}

As pointed out in \cite[Corollary 1.4]{D}, a reduced plane curve $C=V(f)$ with ${\rm mdr}(f)=1$ is either free or nearly free. Thus we have the following.

\begin{corollary}\label{main3}
Let $C=V(f)\subset \PP^2$ be a reduced plane curve of degree $d$ and ${\rm mdr}(f)=1$, let $p\in C$ be an isolated singularity and let $M_f$ be a first syzygy matrix associated with the Jacobian ideal $J_f$ of $f$. 

Then
$\text {rk} (M_f (p))\ge 1 $ if and only if $p$ is a quasi-homogeneous singularity.
\end{corollary}
\begin{proof}
 It immediately follows from Theorems \ref{main1} and \ref{main2}. 
 \end{proof}

 We conclude this section by considering 
curves admitting a linear Jacobian syzygy of the type $(ax_0,bx_1,cx_2)$ for some coefficients $a,b,c \in \CC$. A characterization and a first syzygy matrix of free curves admitting such 
a particular Jacobian syzygy has
been determined in \cite[Theorem 3.5]{BC}. Their result implies 
that if $V(f)$ admits a syzygy of the type $(ax_0,bx_1,cx_2)$
and is not a free curve, hence a nearly free curve, then, up to renaming the variables,
it holds $a = 0$, $bc \neq 0$ and $\partial_0 f\not\in (x_1,x_2)$. This means that $\partial_0 f$ can be written in the form 
\begin{equation}\label{eq: partial_0 f}
    \partial_0 f=g\ x_1+h\ x_2+\omega,
\end{equation}
for some $g \in R_{d-2}$, $h \in \CC[x_0,x_2]_{d-2}$ and a non-zero $\omega \in \CC[x_0]_{d-1}$.
\\

\begin{lemma}
In the expression \eqref{eq: partial_0 f}, one has $h=0$.
\end{lemma}
\begin{proof}
 Suppose by contradiction that 
 $h\neq0$. Since $bx_1\partial_1f+cx_2\partial_2f=0$ one has $bx_1\partial_{01}f+cx_2\partial_{02}f=0$ which implies that
$x_1\mid \partial_{02}f$.
Moreover, since $\partial_0 f=gx_1+hx_2+\omega$, one has $\partial_{02} f=x_1\partial_2g+h+x_2\partial_2 h$. This implies that $x_1\mid (h+x_2\partial_2 h)$ which contradicts $h\in \CC[x_0,x_2]$.
\end{proof}

\begin{proposition}\label{1st_syz_NF} Let $C=V(f)\subset \PP^2$ be a reduced nearly free curve of degree $d$.
Assume that there is a pair $(b,c)$ of non zero elements of $\CC$, such that $(0,bx_1,cx_2)\in \textrm{Syz}(J_f)_1$. 

Then a first and second syzygy matrices $M_f$ and $P_f$ in the minimal free resolution of $J_f$ (see \ref{eq: nearly free sequence}) are given by:

\begin{equation}\label{eq: nearly free sygygies}
M_f=
\begin{pmatrix}0& -\partial_1 f & -\partial_2 f \\
bx_1& \omega& 0\\
cx_2& -\frac{c}{b}gx_2 & \partial_0 f\\
\end{pmatrix}
\qquad \textrm{and }\qquad
P_f=
\begin{pmatrix}\omega\\
-bx_1\\
-cx_2
\end{pmatrix}.
\end{equation}
 \end{proposition}

\section{Examples to illustrate Theorem 5.6}\label{examples}

We analyze in this section a series of examples of reduced plane curves with quasi-homogeneous and non quasi-homogeneous singularities to explicitly demonstrate how our criterion applies.

Observe that it applies also when the coordinates of the singular points are unknown. Indeed, it is enough to check if
the ideal generated by the entries of a first syzygy matrix defines the empty set or not, that is whether its saturation is the irrelevant ideal. Such a check is easily done, for instance, with Macaulay 2 \cite{M2}.

Concerning the nearly free conic-line arrangements with $r=1$, we have the following example (see, for instance, \cite{BMR} and \cite{D2} for more details).


\begin{example}
Let ${\mathcal C}_2$ be a conic arrangement with $m$ conics 
 $C_1,\cdots, C_m$ such that there exist two points $p, q\in \PP^2$ satisfying $C_i\cap C_j=\{p, q\}$ and the two intersection points $p, q$ are tacnodes for $C_i\cup C_j$, for all $i,j$, $1\le i<j\le m$. 

 A similar example in odd degree is given by the conic-line arrangement ${\mathcal C}{\mathcal L}_6$, given by $m$ conics $C_1,\cdots, C_m$ and a line $\ell$ such that there exist two points $p, q\in \PP^2$ such that $C_i\cap C_j=\{p, q\}$ and the two intersection points $p, q$ are tacnodes for $C_i\cup C_j$, for all $i,j$, $1\le i<j\le m$, and $\ell $ is the line through $p$ and $q$.

 Without loss of generality, we can assume that $p=(0:0:1)$, $q=(1:0:0)$ so that the equations of the two curves are given by:
 
\begin{center}
    ${\mathcal C}_2: 
\ f=\prod _{i=1}^m(x_0x_2 + a_ix_1^2) = 0$\hspace{1cm} and
\hspace{1cm}
${\mathcal C}{\mathcal L}_6: 
\ f=x_1\prod _{i=1}^m(x_0x_2 + a_ix_1^2) = 0,
$
\end{center}

\noindent with $a_i\ne 0$ for all $i$, $1\le i \le m$, and $a_i \neq a_j$ if $i \neq j$.

By \cite[Example 3.7]{BC}, for these curves ${\rm Syz}(J_f)_1$ is generated by $(x_0,0,-x_2)$ and they are not free. Since $r=1$, we have $\tau(C)=(d-1)(d-2)$ and they are nearly free by \cite[Lemma 3.4]{BMR}. By Proposition \ref{1st_syz_NF}, in both cases we have $\partial_{1} f\not\in(x_0,x_2)$. It follows that $\partial_1 f=gx_2+\omega$ with $g\in R_{d-2}$ and $\omega\in \mathbb{C}[x_1]_{d-1}$, thus 

\begin{equation}\label{eq: Hilbert Burch C2/CL6}
M_f=
\begin{pmatrix}x_0& \omega & 0 \\
0& -\partial_0 f & \partial_2 f\\
-x_2& gx_2 & -\partial_1 f\\
\end{pmatrix}.
\end{equation}

Therefore $\text {rk}~ M_f(p)\geq 1$ and $\text {rk}~ M_f(q)\geq 1$ and the singularities are quasi-homogeneous as was proved in \cite[Theorem 1.5]{D2}.

 \end{example}

\begin{example}
 Consider the following conic-line arrangement 
$$
C=V(x_0(x_0x_1-x_1^2+x_2^2)(x_0x_2-x_0^2+x_1^2-x_2^2)(x_0x_1+x_1^2-x_2^2)(x_0x_2+x_0^2-x_1^2+x_2^2)),
$$
\noindent which singularities are given by $p_1=(1:0:0),~p_2=(1:-1:0),~p_3=(1:1:0),~p_4=(0:-1:1)\textrm{ and } p_5=(0:1:1)$.

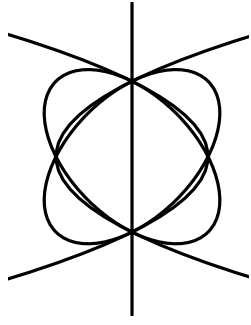
\begin{center}\begin{tikzpicture}[line cap=round,line join=round,>=triangle 45,x=1cm,y=1cm]
\clip(-1.6234528862232285,-2.102782204314926) rectangle (1.532151291942132,2.042357111201054);
\draw [line width=1.2pt] (0,-2.102782204314926) -- (0,2.042357111201054);
\draw [rotate around={-135:(0,0)},line width=1.2pt] (0,0) ellipse (1.4142135623730951cm and 0.816496580927726cm);
\draw [samples=50,rotate around={-270:(1,0)},xshift=1cm,yshift=0cm,line width=1.2pt,domain=-3:3)] plot (\x,{(\x)^2/2/0.5});
\draw [samples=50,rotate around={-90:(-1,0)},xshift=-1cm,yshift=0cm,line width=1.2pt,domain=-3:3)] plot (\x,{(\x)^2/2/0.5});
\draw [rotate around={-45:(0,0)},line width=1.2pt] (0,0) ellipse (1.4142135623730951cm and 0.816496580927726cm);
\begin{scriptsize}
\draw [fill=ttqqqq] (0,-1) circle (2.5pt);
\draw[color=ttqqqq] (0.26,-1.45) node {$p_4$};
\draw [fill=ttqqqq] (0,1) circle (2.5pt);
\draw[color=ttqqqq] (0.26,1.45) node {$p_5$};
\end{scriptsize}
\end{tikzpicture}
\end{center}

Computing the minimal free $R$-resolution of the Jacobian ideal of $C$ one has:
$$
0 \longrightarrow R(-14) \longrightarrow R(-13)^2\oplus R(-12)\overset {M_f}{\longrightarrow} R(-8)^3 \longrightarrow J_f \longrightarrow 0,
$$
Thus $C$ is nearly free with ${\rm mdr}(f)=4$, $\deg(J_f)=\tau(C)=47$, and $M_f$ is given by
\begin{equation}\label{eq: 2 NQH}
M_f=\begin{pmatrix} 
A_0 &B_0 &C_0\\
A_1 &B_1 &C_1\\
A_2 &B_2 &C_2\\
\end{pmatrix} ,
 \end{equation}
 
\noindent with 
$$\begin{array}{rcl}
  A_0 & = & 2x_0^3x_1-8x_0x_1^3+8x_0x_1x_2^2, \\
  B_0 & = & 4x_0^5-16x_0^3x_1^2+6x_0^3x_2^2-8x_0x_1^2x_2^2+8x_0x_2^4, \\
  C_0 & = & -16x_0^3x_1x_2+16x_0x_1^3x_2-16x_0x_1x_2^3,\\
  A_1 & = & -7x_0^2x_1^2+x_1^4+26x_1^2x_2^2-27x_2^4,\\
  B_1 & = & -14x_0^4x_1+2x_0^2x_1^3+6x_0^2x_1x_2^2+28x_1^3x_2^2-28x_1x_2^4, \\
  C_1 & = & -25x_0^2x_1^2x_2+25x_1^4x_2+54x_0^2x_2^3-106x_1^2x_2^3+81x_2^5, \\
  A_2 & = & -16x_0^2x_1x_2+28x_1^3x_2-28x_1x_2^3, \\
  B_2 & = & -5x_0^4x_2+2x_0^2x_1^2x_2+27x_1^4x_2-3x_0^2x_2^3-26x_1^2x_2^3-x_2^5, \text{ and }\\
  C_2 & = & 27x_0^4x_1-54x_0^2x_1^3+27x_1^5+101x_0^2x_1x_2^2-110x_1^3x_2^2+83x_1x_2^4.
\end{array}
$$  

With the help of Macaulay 2, we can check that $\text {rk}~ M_f(p_1)=\text {rk}~ M_f(p_2)=\text {rk}~ M_f(p_3)=1$ and these points are quasi-homogeneous singularities, meanwhile $\text {rk}~ M_f(p_4)=\text {rk}~ M_f(p_5)=0$, and these points are not quasi-homogeneous.

\vskip 2mm
Next we present another example of a plane curve with multiple non quasi-homogeneous singularities.
\end{example}
\begin{example}
  \label{5non QH} Consider the following conic-line arrangement given by one smooth conic and ten lines built from \cite[Example 1.7]{ST}:

\begin{center}
\begin{tikzpicture}
\filldraw [black] (0,0) circle [radius=2pt]
(0,1) circle [radius=2pt]
(1,2) circle [radius=2pt]
(2,1) circle [radius=2pt]
(1,0) circle [radius=2pt];

\draw (0,-0.5) -- (0,1.5);
 \draw (-0.5,0) -- (2.5,0);
 \draw (-0.3,-0.6) -- (1.2,2.4);
 \draw (-0.6,-0.3) -- (2.4,1.2);
 \draw (1,-0.5) -- (1,2.5);
 \draw (-0.5,1) -- (2.5,1);
 \draw (-0.5,0.5) -- (1.5,2.5);
 \draw (0.5,-0.5) -- (2.5,1.5);
 \draw (-0.5,1.5) -- (1.5,-0.5);
 \draw (0.5,2.5) -- (2.5,0.5);
\draw[xshift=28.3465pt,yshift=28.3465pt,rotate=45,thick] (0,0) ellipse (1.4 and 0.8);
 \end{tikzpicture}
\end{center}

More precisely, we take 
\[
\begin{array}{rr}
& C=V(x_0x_1(x_0-2x_1)(x_1-2x_0)(x_0-x_2)(x_1-x_2)(x_0-x_1+x_2)(x_0-x_1-x_2) \\
& (x_0+x_1-x_2)
 (x_0+x_1-3x_2)(x_0^2+x_1^2-x_0x_1-x_0x_2-x_1x_2)).
\end{array}
\]

  The curve $C$ has twenty singular points, five of them $(0:0:1),(0:1:1),(1:2:1),(2:1:1),(1:0:1)$ are given by the intersection of the conic with four lines, and the other fifteen are given by the intersection of two lines, therefore all the singularities are ordinary and the Milnor number can be easily computed. One has $\mu_p(C)=16$ in each of the five singular points on the conic and $\mu_p(C)=1$ in the other fifteen. Thus the total Milnor number $\mu(C)=95$.

 Again with the help of Macaulay 2, we can compute the minimal free $R$-resolution of the Jacobian ideal of $C$, given by:
$$0 \longrightarrow R(-18) \longrightarrow R(-17)^3 \overset {M_f}{\longrightarrow} R(-11)^3 \longrightarrow J_f \longrightarrow 0,$$
thus $C$ is nearly free with ${\rm mdr}(f)=6$, $\deg(J_f)=\tau(C)=90$.
 
 Since $\tau(C)<\mu(C)$ we know that at least one of the singularities is non quasi-homogeneous. Indeed, $\text {rk}~ M_f(p)=0$ for the five singular points on the conic, thus all of them are non quasi-homogeneous singularities. For the other fifteen as expected $\text {rk}~ M_f(p)=1$, and they are quasi-homogeneous singularities.

\end{example}

\section{Final comments}
In this short last section, we state a natural question and open problem stemming from our 
work.

Consider a reduced plane curve $C=V(f)$ of degree $d$ and a minimal free resolution of its Jacobian ideal:
$$
0\longrightarrow \bigoplus_{j=1}^{m-2} R(-e_j)\overset {P_f}{\longrightarrow} \bigoplus _{i=1}^mR(-d_i) \overset {M_f}{\longrightarrow} R^3\longrightarrow J_f(d-1) \longrightarrow 0.
$$
The matrix $M_f$ will be called a {\em first syzygy matrix} of the Jacobian ideal $J_f$ of $f$. We ask: 
\begin{question}
 Can we characterize the quasi-homogeneous singular points of $C$ in terms of $M_f$?
\end{question}

In the particular case of 3-syzygy curves, that is
curves with a syzygy matrix $M_f$ a square matrix of size $3$, by the evidence given by several examples that we have tested, we state the following conjecture:
\begin{conjecture}\label{conj}
Let $C=V(f)$ be a reduced 3-syzygy plane curve with a
first syzygy matrix $M_f$. A point $p\in Sing( C )$ is a quasi-homogeneous isolated singularity if and only if $\textrm{rk}~ M_f(p)\ge 1$.
\end{conjecture}

We end with an example that supports our conjecture.

\begin{example}\label{3syzygy} Consider the rational irreducible plane curve $$C =V(f) =V(x_0^d + (x_0^2 + x_1^2)^kx_2),$$
with $d=2k+1$ odd, $k \geq 2$. According to \cite[Example 4.4]{DS6} $C$ is a $3$-syzygy curve with exponents $(2, d-2, d-1)$, only one singular point $p=(0:0:1)$, with $\mu_p(C)= d^2-3d+1$, $\tau_p(C)=d^2-4d+5$, thus $p$ is non quasi-homogeneous. And moreover, if $k$ is odd, one has:
\begin{equation}\label{eq: Hilbert Burch 3-syz}
M_f=
\begin{pmatrix}
0& 2kx_2(x_0^2+x_1^2)^{k-1} & (x_0^2+x_1^2)^{k} \\
(x_0^2+x_1^2)& (2k+1)x_1^{2k-1} & 0\\
-2kx_1x_2&-B & -((2k+1)x_0^{2k}+2kx_0x_2(x_0^2+x_1^2)^{k-1})\\
\end{pmatrix},
\end{equation}

with $$B=2k(2k+1)x_2(x_0^{2k-2}-x_0^{2k-4}x_1^2+ \cdots -x_0^2x_1^{2k-4}+4k^2x_0x_2^2(x_0^2+x_1^2)^{k-2}).$$
\end{example}
Therefore, $M_f(p)=0$ as we conjectured. The case $k$ even is analogous.

\end{document}